\documentclass[11pt]{article} 
\usepackage{graphicx}
\usepackage{latexsym}
\usepackage{amsmath}
\usepackage{amssymb}
\usepackage{amsthm}
\usepackage{verbatim}
\usepackage{tikz}
\usepackage{bm}

\newcommand{\CP}{{\cal P}}
\newcommand{\CN}{{\cal N}}

\title{Forbidden-Total-Size Nim}
\author{Hiromi Oginuma\ \thanks{Nara Women's University.
{\tt xah\_oginuma@cc.nara-wu.ac.jp}} \and
Masato Shinoda \thanks{Nara Women's University. {\tt shinoda@cc.nara-wu.ac.jp}}}

\date{\today}

\begin{document}

\theoremstyle{definition} 
\newtheorem{theorem}{Theorem}[section]
\newtheorem{definition}[theorem]{Definition}
\newtheorem{lemma}[theorem]{Lemma}
\newtheorem{proposition}[theorem]{Proposition}
\newtheorem{corollary}[theorem]{Corollary}
\newtheorem{example}[theorem]{Example}
\newtheorem{remark}[theorem]{Remark}
\newtheorem{conjecture}[theorem]{Conjecture}
\def\theequation{\thesection.\arabic{equation}}

\maketitle              
\begin{abstract}
We consider a variant of Nim in which, for a fixed set $S$ of nonnegative integers, a move is forbidden if the total number of remaining stones belongs to $S$. This game coincides with normal-play Nim when $S=\emptyset$, and with mis\`ere Nim when $S=\{0\}$. In this paper, we focus in particular on the case where $S$ is the set of multiples of $d$, and show that simple criteria for determining the outcome can be obtained for $d=2,3,4$.
\end{abstract}

\section{Introduction}

Nim is one of the best-known classical combinatorial games. In its most basic form, the game consists of three piles of stones. Two players take turns, and on each turn, a player chooses one pile and removes any positive number of stones from it. A position of the game can be represented by a triple $(x,y,z)$, where $x,y,z$ are nonnegative integers denoting the numbers of stones in the three piles. For example, from the position $(3,5,7)$, the player to move can reach $(3,2,7)$ or $(0,5,7)$ in one move.

The rule under which the player who removes the last stone wins, or equivalently, the player who has no legal move on their turn loses, is called {\it normal} play. In contrast, the rule under which the player who removes the last stone loses, or equivalently, the player who has no legal move on their turn wins, is called {\it mis\`ere} play. For example, from the position $(3,0,0)$, the player to move can win by moving to $(0,0,0)$ under normal play, and by moving to $(1,0,0)$ under mis\`ere play. These games can be defined in the same way when the number of piles is greater than three.

Mis\`ere Nim can also be viewed as a game in which moves that leave no stones are forbidden, and the player who has no legal move on their turn loses. Motivated by this observation, we introduce a new variant as follows. Fix a subset $S$ of the nonnegative integers, and forbid any move that results in a position in which the total number of remaining stones belongs to $S$. The player who has no legal move on their turn loses. We call this game {\bf Forbidden-Total-Size Nim} ({\bf FTS Nim}). When $S=\emptyset$, the game coincides with normal-play Nim, while when $S=\{0\}$, it coincides with mis\`ere Nim.

As a special case of FTS Nim, we focus in this paper on the case where $S$ is the set of multiples of $d$, and call the resulting game $d$-FTS Nim. In $d$-FTS Nim, we assume that $d$ is an integer with $d\geq 2$. Formally, however, if we set $d=0$, the resulting game can be regarded as mis\`ere Nim.

As a first step in the study of $d$-FTS Nim, in Section 2 we give a criterion for determining the outcome in the case of two piles (Theorem 2.3). In Section 3, for each of $d=2,3,4$, we give a criterion for determining the outcome for $n$ piles (Theorems 3.1, 3.2, and 3.3). For $d\geq 5$, however, it is difficult to obtain such a general criterion. We therefore restrict our attention to three piles in Section 4 and establish several results that hold for general $d$. When $d=2^k$, we are able to obtain an explicit criterion for determining the outcome (Theorem 4.4). Finally, in Section 5, returning to the original definition of FTS Nim, we present some preliminary results for the case $S=\{0,1\}$.

\section{Basic Properties of Forbidden-Total-Size Nim}

\subsection{Sets of $\CP$-positions under Normal Play and Mis\`ere Play}

In this section, we describe the rules of FTS Nim and present the result for the simplest case, in which there are two piles. Before doing so, we review normal play and mis\`ere play, which are the basic forms of Nim, and introduce the nim-sum, an important tool for the analysis of Nim.

Let the number of piles be $n$, where $n\geq 2$, and let the set of positions be
$\mathbb{X}=\{\bm{x}=(x_1,x_2,\ldots,x_n)\mid x_i \text{ is a nonnegative integer for each } i\}$.  
For $\bm{x}=(x_1,x_2,\ldots,x_n)\in\mathbb{X}$, let $|\bm{x}|=\sum_{i=1}^{n}x_i$
denote the total number of stones in the position $\bm{x}$. Two players take turns, and on each turn, a player chooses one pile and removes any positive number of stones from it. 
If a position $\bm{x}$ can be changed to a position $\bm{x}'$ in a single move, we write $\bm{x}\to\bm{x}'$. If this move removes stones from the $i$th pile, changing its size from $x_i$ to $x'_i$, we write $x_i\to x'_i$; in this case, $x_i>x'_i$. 
The players alternate turns, and the game ends when no stones remain.

There are two ways to determine the winner. Under normal play, the player who removes the last stone wins, whereas under mis\`ere play, the player who removes the last stone loses.
Since these games are finite two-player zero-sum games of perfect information with no draws, every position can be classified as either an $\CN$-position, in which the player to move has a winning strategy, or a $\CP$-position, in which the player who is not to move has a winning strategy.

To classify positions in Nim, let $a$ and $b$ be nonnegative integers with binary representations
$a=\sum_i 2^i a_i, b=\sum_i 2^i b_i$.
The nim-sum (exclusive OR) of $a$ and $b$, denoted by $a\oplus b$, is defined by
$a\oplus b=\sum_i 2^i((a_i+b_i)\bmod 2)$,
that is, as addition in binary without carrying. Since the nim-sum is associative, so that
$(a\oplus b)\oplus c=a\oplus(b\oplus c)$,
it can be defined analogously for three or more nonnegative integers. Using the nim-sum, it is well known that the condition for a position $\bm{x}\in\mathbb{X}$ to be a $\CP$-position can be expressed as follows.

\begin{theorem}[Set of $\CP$-positions in Nim]
Define the following subsets of $\mathbb{X}$:
\begin{eqnarray*}
P_{N}&=& \{ \bm{x}\in\mathbb{X} \mid x_{1}\oplus x_{2}\oplus\cdots\oplus x_{n}=0\},\\
P_{M}&=& \{\bm{x}\in\mathbb{X} \mid \ x_{i}\geq 2
\mbox{ for some } i \mbox{ and } x_{1}\oplus x_{2}\oplus\cdots\oplus x_{n}=0\}\\
&&\cup \ \{\bm{x}\in\mathbb{X} \mid 0\leq x_{i}\leq 1
\mbox{ for each } i \mbox{ and } |\bm{x}| \mbox{ is odd}\}.
\end{eqnarray*}
Then $P_{N}$ is the set of $\CP$-positions in normal-play Nim, and $P_{M}$ is the set of $\CP$-positions in mis\`ere Nim.
\end{theorem}

These results were obtained by Bouton \cite{bou02}. In particular, the following two properties of the nim-sum will be used repeatedly. 
For a position
$\bm{x}=(x_1,x_2,\ldots,x_n)$, the following hold.
\begin{itemize}
\item[(N1)] If $x_{1}\oplus x_{2}\oplus\cdots\oplus x_{n}=0$, then removing stones from any pile, that is, for any move $x_i\to x'_i$, yields $x_{1}\oplus x_{2}\oplus\cdots\oplus x'_i\oplus\cdots\oplus x_{n}\neq 0$.
\item[(N2)] If
$x_{1}\oplus x_{2}\oplus\cdots\oplus x_{n}\neq 0$,
then it is possible to remove an appropriate number of stones from some pile, that is, there exists a move $x_i\to x'_i$, such that $x_{1}\oplus x_{2}\oplus\cdots\oplus x'_i\oplus\cdots\oplus x_{n}=0$.
\end{itemize}
See also Berlekamp, Conway and Guy \cite{ber01}, and Siegel \cite{sie13}.

\subsection{Forbidden-Total-Size Nim}

We interpret mis\`ere Nim as a game in which moves that leave no stones are forbidden and the player who has no legal move on their turn loses, and consider the following extension of this rule. Let $S$ be a subset of the nonnegative integers, and define $\mathbb{X}_S=\{\bm{x}\in\mathbb{X}\mid |\bm{x}|\in S\}$, 
the set of positions in which the total number of remaining stones belongs to $S$.

\begin{definition}[Forbidden-Total-Size Nim (FTS Nim)]
There are $n$ piles of stones, and two players take turns. On each turn, a player chooses one pile and removes any positive number of stones from it. However, any move resulting in a position in $\mathbb{X}_S$ is forbidden. The player who has no legal move on their turn loses. We call the game under this rule Forbidden-Total-Size Nim (FTS Nim).
\end{definition}

Note that, under this rule, no position in $\mathbb{X}_S$ can occur during the course of a game, although we allow a game to start from a position in $\mathbb{X}_S$. Throughout this paper, we use $P_0$ to denote the set of $\CP$-positions contained in $\mathbb{X}_S$, $P_1$ to denote the set of $\CP$-positions not contained in $\mathbb{X}_S$, and $P$ to denote the set of all $\CP$-positions, so that $P=P_0\cup P_1$.

The aim of this study is to give a criterion for determining whether a given position $\bm{x}\in\mathbb{X}$ is a $\CP$-position or an $\CN$-position. In this paper, we focus in particular on the case where $d\geq 2$ is an integer and $S$ is the set of multiples of $d$, and call the resulting game {\bf $d$-FTS Nim}. For $d$-FTS Nim, we write $\mathbb{X}_d = \{\bm{x}\in\mathbb{X}\mid |\bm{x}|\equiv 0 \pmod d\}$ for the set of positions in which the total number of remaining stones is a multiple of $d$. Thus, moves to positions in $\mathbb X_d$ are forbidden. The set of terminal positions that can occur during the course of a game is $\{\bm{x}\in\mathbb{X} \ | \ x_{i}\leq 1 \mbox{ for each }i \mbox{ and } |\bm{x}|\equiv 1\ ({\rm mod }\  d)\}$. 

A variant of Nim in which forbidden positions are specified in a similar way was studied by Garrabrant, Friedman, and Landsberg \cite{gar13}. However, their framework does not include the case where $S$ is an infinite set, as considered in the present study. Recently, Tyagi \cite{tya26} introduced Multiplicative Modular Nim, a variant of Nim based on modular arithmetic in which the product of the pile sizes modulo a fixed integer plays a central role. Unlike the present game, its modular condition is imposed on the individual pile sizes rather than on the total number of stones.

\subsection{Outcome Characterization for Two Piles}

As a first step in the study of $d$-FTS Nim, we consider the case of two piles, that is, $n=2$, and determine the set of $\CP$-positions.
First, when $d=2$, $P_0=\{(0,0),(2,2),(4,4),\ldots\}$, and $P_1=\{(0,1),(1,0),(2,3),(3,2),(4,5),(5,4),\ldots\}$.
This result is included in Theorem 3.1 in Section 3.

The following theorem completely characterizes the two-pile case. Let $D_d$ denote the set of nonnegative integers $x$ such that $2x$ is a multiple of $d$, that is, $D_d=\{x \ | \ 2x\equiv 0 \ (\mbox{mod}\ d)\}$.

\begin{theorem}
Consider two-pile $d$-FTS Nim with $d\geq 3$. Define the following subsets of $\mathbb{X}$:
\begin{eqnarray*}
P_0&=&\{(x,x) \mid x\in D_d\}, \\
P_{1,0}&=&\{(x,x) \mid x\notin D_d\mbox{ and }x-1\notin D_d\},\\
P_{1,1}&=&\{(x,x+1),(x+1,x) \mid x\in D_d\}.
\end{eqnarray*}
Let $P_1=P_{1,0}\cup P_{1,1}, P=P_0\cup P_1$. Then $P$ is the set of $\CP$-positions of two-pile $d$-FTS Nim.
\end{theorem}

\medskip

For example, when $d=3$, $P_0=\{(0,0),(3,3),(6,6),\ldots\}, P_{1,0}=\{(2,2),$\\$(5,5),(8,8),\ldots\}$, and $P_{1,1}=\{(0,1),(1,0),(3,4),(4,3),(6,7),(7,6),\ldots\}$.

\medskip

{\bf Proof.}
Since $P_0$ is a subset of $\mathbb{X}_d$ and therefore contains no position that can occur during the course of a game, it suffices to prove the following two statements.

\begin{itemize}
\item[(i)] If $\bm{x}\in P$ and $\bm{x}\to\bm{x'}$, then $\bm{x'}\notin P_1$.
\item[(ii)] For each $\bm{x}\notin P$, there exists $\bm{x'}$ such that
$\bm{x}\to\bm{x'}$ and $\bm{x'}\in P_1$.
\end{itemize}

(i) Suppose that $(x,x)\in P_0$. After stones are removed from one of the two piles, the two pile sizes become different, so the resulting position cannot belong to $P_{1,0}$. Moreover, even if the move is to $(x,x-1)$, we have $x-1\notin D_d$ because $x\in D_d$; here the assumption $d\neq 2$ is required. Hence the resulting position cannot belong to $P_{1,1}$ either.

Similarly, if $(x,x)\in P_{1,0}$, then $x-1\notin D_d$, and therefore $(x-1,x)\notin P_{1,1}$. Finally, suppose that $(x,x+1)\in P_{1,1}$. Since $x\in D_d$, we have $(x,x)\notin P_{1,0}$, and since $x-1\notin D_d$, we also have $(x,x-1)\notin P_{1,1}$.

(ii) Let $(x,y)\notin P$. Without loss of generality, we may assume that either $x=y$ and $x-1\in D_d$, or $x<y$.

If $x=y$ and $x-1\in D_d$, we can move to $(x-1,x)\in P_{1,1}$. If $x<y$ and $x,x-1\notin D_d$, we can move from $(x,y)$ to $(x,x)\in P_{1,0}$. If $x<y$ and $x\in D_d$, then $x+1\neq y$ because $(x,y)\notin P_{1,1}$, and hence we can move to $(x,x+1)\in P_{1,1}$. Finally, if $x<y$ and $x-1\in D_d$, we can move to $(x,x-1)\in P_{1,1}$.
\qed

\section{$n$-Pile $d$-FTS Nim for $d\leq 4$}

In this section, we consider the case where the number of piles is $n\geq 3$. In particular, for $d=2,3,4$, the set of $\CP$-positions can be determined as described below.

From this point on, let $\bm{x}=(x_1,x_2,\ldots,x_n), \bm{x'}=(x'_1,x'_2,\ldots,x'_n)$.
For an integer $m\geq 2$, define
\[
B_m=
\left\{
\bm{x}\in\mathbb{X}
\ \middle|\
\left\lfloor\frac{x_1}{m}\right\rfloor
\oplus
\left\lfloor\frac{x_2}{m}\right\rfloor
\oplus\cdots\oplus
\left\lfloor\frac{x_n}{m}\right\rfloor
=0
\right\}.
\]
Thus, $B_m$ is the set of positions for which the nim-sum of the quotients obtained by dividing the sizes of the piles by $m$ is zero.

\subsection{The Case $d=2$}

In $2$-FTS Nim, moves to positions with an even total number of stones are forbidden. The set of $\CP$-positions is characterized as follows.

\begin{theorem}
Consider $2$-FTS Nim, that is, $d$-FTS Nim with $d=2$. Define
$P=P_0\cup P_1$,
where
\begin{eqnarray*}
P_0
&=&
\{\bm{x}\in B_2\cap\mathbb{X}_2
\mid
x_i \mbox{ is even for each } i\},
\\
P_1
&=&
\{\bm{x}\in B_2\setminus\mathbb{X}_2\}
\\
&=&
\{\bm{x}\in B_2
\mid
|\bm{x}| \mbox{ is odd}\}.
\end{eqnarray*}
Then $P$ is the set of $\CP$-positions of $2$-FTS Nim.
\end{theorem}

\medskip

{\bf Proof.}
(i) We first show that if $\bm{x}\in P_0$ and $\bm{x}\to\bm{x'}$, then $\bm{x'}\notin P_1$.
Since each $x_i$ is even, whenever $x_i\to x'_i$ we have $\lfloor\frac{x_i}{2}\rfloor
\neq\lfloor\frac{x'_i}{2}\rfloor$.
Thus, if $\bm{x}\in B_2$, property (N1) of the nim-sum implies that
$\bm{x'}\notin B_2$, and hence $\bm{x'}\notin P_1$.

(ii) Next, we show that if $\bm{x}\in P_1$ and $\bm{x}\to\bm{x'}$, then $\bm{x'}\notin P_1$.
For
$\lfloor\frac{x_i}{2}\rfloor=\lfloor\frac{x'_i}{2}\rfloor$
to hold after a move $x_i\to x'_i$, it is necessary that $x'_i=x_i-1$.
In this case exactly one stone is removed, so $|\bm{x}|$ and $|\bm{x'}|$ have opposite parity. Since $|\bm{x}|$ is odd, $|\bm{x'}|$ is even. Therefore, $\bm{x'}\notin P_1$.

(iii) Finally, we show that if $\bm{x}\notin P$, then there exists a position $\bm{x'}$ such that
$\bm{x}\to\bm{x'}$ and $\bm{x'}\in P_1$.

Suppose first that $\bm{x}\in B_2$. Since $\bm{x}\notin P$, $|\bm{x}|$ must be even; otherwise, $\bm{x}$ would belong to $P_1$. Moreover, since $\bm{x}\notin P_0$, there exists some $x_j$ that is odd. If we make the move
$x_j\to x'_j=x_j-1$,
then $\lfloor\frac{x_j}{2}\rfloor=\lfloor\frac{x'_j}{2}\rfloor$,
and hence $\bm{x'}\in B_2$. Furthermore, $|\bm{x'}|$ is odd, so $\bm{x'}\in P_1$.

Now suppose that $\bm{x}\notin P$ and $\bm{x}\notin B_2$. By property (N2) of the nim-sum, we can change some
$\lfloor\frac{x_j}{2}\rfloor$ 
to $\lfloor\frac{x'_j}{2}\rfloor$
so that $\bm{x'}\in B_2$. There are two possible values of $x'_j$ satisfying this condition: if
$\lfloor\frac{x'_j}{2}\rfloor=z$,
then $x'_j=2z$ or $x'_j=2z+1$. We therefore choose the one for which $|\bm{x'}|$ is odd. Then $\bm{x'}\in P_1$.
\qed

\subsection{The Case $d=3$}

We next consider $3$-FTS Nim. In this case, the set of $\CP$-positions is characterized as follows.

\medskip

\begin{theorem}
Consider $3$-FTS Nim, that is, $d$-FTS Nim with $d=3$. Define
$P=P_0\cup P_1$, where
\begin{eqnarray*}
P_0
&=&
\{\bm{x}\in B_3\cap\mathbb{X}_3
\mid
x_i\not\equiv 2 \!\! \pmod 3 \mbox{ for each } i\},
\\
P_1
&=&
\{\bm{x}\in B_3\setminus\mathbb{X}_3
\mid
|\bm{x}|\equiv 1 \!\! \pmod 3\}.
\end{eqnarray*}
Then $P$ is the set of $\CP$-positions of $3$-FTS Nim.
\end{theorem}

\medskip

{\bf Proof.}
(i) We first show that if $\bm{x}\in P_0$ and $\bm{x}\to\bm{x'}$, then $\bm{x'}\notin P_1$.
Since $\bm{x}\in B_3$, property (N1) of the nim-sum implies that, in order for $\bm{x'}$ to belong to $B_3$, we must have
$\lfloor\frac{x_i}{3}\rfloor=\lfloor\frac{x'_i}{3}\rfloor$
for each $i$.
Therefore, if the move is $x_j\to x'_j$, then, since
$x_j\not\equiv 2\pmod 3$, it is necessary that
$
x_j\equiv 1\pmod 3
$ and $x'_j=x_j-1$. 
In this case exactly one stone is removed. Since $\bm{x}\in\mathbb{X}_3$, we have
$
|\bm{x'}|\equiv 2\pmod 3,
$
and hence $\bm{x'}\notin P_1$.

(ii) Next, we show that if $\bm{x}\in P_1$ and $\bm{x}\to\bm{x'}$, then $\bm{x'}\notin P_1$.
Since $\bm{x}\in B_3$, as in (i), for $\bm{x'}$ to belong to $B_3$ we must have
$
\lfloor\frac{x_i}{3}\rfloor=\lfloor\frac{x'_i}{3}\rfloor$
for each $i$.
If $x_j\to x'_j$ and
$
\lfloor\frac{x_j}{3}\rfloor=\lfloor\frac{x'_j}{3}\rfloor,
$
then $x'_j=x_j-1$ or $x'_j=x_j-2$. Hence
$|\bm{x}|\not\equiv |\bm{x'}|\pmod 3$.
Since $|\bm{x}|\equiv 1\pmod 3$, it follows that $\bm{x'}\notin P_1$.

(iii) Finally, we show that if $\bm{x}\notin P$, then there exists a position $\bm{x'}$ such that
$\bm{x}\to\bm{x'}$ and $\bm{x'}\in P_1$.

First suppose that $\bm{x}\in B_3\cap\mathbb{X}_3$. Since $\bm{x}\notin P$, there exists some $x_j$ such that
$x_j\equiv 2\pmod 3$.
If we set
$x'_j=x_j-2$,
then
$\lfloor\frac{x_j}{3}\rfloor=\lfloor\frac{x'_j}{3}\rfloor$,
so $\bm{x'}$ remains in $B_3$, while
$|\bm{x'}|\equiv 1\pmod 3$.
Thus $\bm{x'}\in P_1$.

Now suppose that $\bm{x}\notin B_3$. By property (N2) of the nim-sum, we can change some
$\lfloor\frac{x_j}{3}\rfloor$ to $\lfloor\frac{x'_j}{3}\rfloor$ so that $\bm{x'}\in B_3$. There are three possible values of $x'_j$ satisfying this condition: if $\lfloor\frac{x'_j}{3}\rfloor=z$, then $x'_j=3z, 3z+1$, or $3z+2$.
 We therefore choose the one for which $|\bm{x'}|\equiv 1\pmod 3$. Then $\bm{x'}\in P_1$.
\qed

\subsection{The Case $d=4$}

We next consider $4$-FTS Nim.

\begin{theorem}
Consider $4$-FTS Nim, that is, $d$-FTS Nim with $d=4$. Define
$P=P_0\cup P_1$, where
\begin{eqnarray*}
P_0
&=&
\{\bm{x}\in B_2\cap\mathbb{X}_4
\mid
x_i \mbox{ is even for each } i\},
\\
P_1
&=&
\{\bm{x}\in B_2\setminus\mathbb{X}_4
\mid
|\bm{x}| \mbox{ is odd}\}.
\end{eqnarray*}
Then $P$ is the set of $\CP$-positions of $4$-FTS Nim.
\end{theorem}

Note that if $\bm{x}\in B_2$, then the nim-sum of the $2^1$-digits of the pile sizes is zero. Moreover, if all $x_i$ are even, then $|\bm{x}|$ is necessarily a multiple of $4$. Therefore, the conditions defining the positions in $P_0$ and $P_1$ are exactly the same as those in the case $d=2$. Hence, the theorem follows by the same proof as in the case $d=2$.

\section{$d$-FTS Nim for $d\geq 5$}

In the previous section, we completely determined the set of $\CP$-positions for $d\leq 4$. In contrast, determining the set of $\CP$-positions appears to be more difficult for $d\geq 5$. In what follows, we first explain that even for three piles, a necessary and sufficient condition for a position to be a $\CP$-position has not yet been obtained when $d=5$. We then show that when $d=2^k$ with $k\geq 3$ and there are three piles, a structure similar to those for $d=2$ and $d=4$ arises, allowing the set of $\CP$-positions to be determined.

\subsection{The Case $d=5$}

For $5$-FTS Nim, by analogy with the case $d=3$, one might expect that the set of $\CP$-positions could be characterized by a condition such as requiring the nim-sum of the quotients obtained by dividing the pile sizes by $5$ to be zero. However, there are $\CP$-positions, such as $(5,4,2)$, that do not satisfy this quotient nim-sum condition. Thus, the set of $\CP$-positions cannot be characterized in the same form as for $d\leq 4$.

We therefore first restrict our attention to the three-pile case. As a step toward determining the set of $\CP$-positions, we show that for each pair $(x,y)$, there exists a unique $z$ such that $(x,y,z)$ is a $\CP$-position belonging to $\mathbb{X}\setminus\mathbb{X}_d$. Recall that $P_1$ denotes the set of $\CP$-positions not contained in $\mathbb{X}_d$.

\medskip

\begin{theorem}
Consider three-pile $d$-FTS Nim with $d\geq 2$. For every pair $(x,y)$ of nonnegative integers, there exists a unique $z$ such that $(x,y,z)\in P_1$.
\end{theorem}

\medskip

{\bf Proof.}
We prove the uniqueness and existence of such a $z$ separately.

First, suppose that for some pair $(x,y)$,
$(x,y,z)\in P_1$ and $(x,y,z')\in P_1$.
If $z>z'$, then there is a move from $(x,y,z)$ to $(x,y,z')$, which is impossible since both positions belong to $P_1$. Hence $z=z'$, proving uniqueness.

Next, suppose that for some pair $(x,y)$, there is no $z$ such that $(x,y,z)\in P_1$. Then, for each $z$ such that
$(x,y,z)\notin\mathbb{X}_d$,
there must exist either some $x'<x$ such that
$(x',y,z)\in P_1$, or some $y'<y$ such that $(x,y',z)\in P_1$.
The pairs $(x',y)$ and $(x,y')$ arising in this way must be distinct for different values of $z$. Indeed, for example, if
$(x',y,z)\in P_1, (x',y,z')\in P_1$, and $z\neq z'$, then this would contradict the uniqueness proved above.
 However, there are only finitely many possible pairs of the form $(x',y)$ with $x'<x$ or $(x,y')$ with $y'<y$, whereas there are infinitely many possible values of $z$. This is a contradiction. Therefore, for every pair $(x,y)$, there exists a $z$ such that $(x,y,z)\in P_1$.
\quad$\Box$

\medskip

By this theorem, for each pair $(x,y)$, we can define $f(x,y)$ to be the unique value of $z$ such that $(x,y,z)\in P_1$.
Note that if positions in $\mathbb{X}_d$ are also allowed, there may be more than one value of $z$ for which $(x,y,z)\in P$. For example, in $5$-FTS Nim, both $(5,0,5)$ and $(5,0,6)$ are $\CP$-positions. Therefore, $f(x,y)$ is defined so that
$(x,y,f(x,y))\notin\mathbb{X}_d$.

For $d=2$, it follows from Theorem 3.1 that if
$x=2q_{x}+\varepsilon_x, y=2q_{y}+\varepsilon_y$,
where $q_{x}$ and $q_{y}$ are the quotients obtained by dividing $x$ and $y$, respectively, by $2$, then
\begin{equation}
f(x,y)=2(q_{x}\oplus q_{y})+\varepsilon,
\end{equation}
where $\varepsilon\in\{0,1\}$ and
$\varepsilon_x+\varepsilon_y+\varepsilon$ is odd.

Similarly, for $d=3$, it follows from Theorem 3.2 that if $x=3q_{x}+\varepsilon_x, y=3q_{y}+\varepsilon_y$,
where $q_{x}$ and $q_{y}$ are the quotients obtained by dividing $x$ and $y$, respectively, by $3$, then
\begin{equation}
f(x,y)=3(q_{x}\oplus q_{y})+\varepsilon,
\end{equation}
where $\varepsilon\in\{0,1,2\}$ and $\varepsilon_x+\varepsilon_y+\varepsilon\equiv 1\pmod 3$.

For a general integer $d\geq 2$, the values of $f(x,y)$ can be determined recursively, starting with positions having smaller total numbers of stones. For a proper subset $T$ of the nonnegative integers, let ${\rm mex}(T)$ denote the smallest nonnegative integer not contained in $T$ (the Minimum EXcluded value).

\begin{proposition}
We have $f(0,0)=1$, and for $x+y\geq 1$, $f(x,y)$ is recursively given by
\begin{equation}
f(x,y)={\rm mex}\left\{
\begin{array}{ll}
f(x',y) & (0\leq x' < x), \\[2mm]
f(x,y') & (0\leq y' < y), \\[2mm]
kd-(x+y) & \left(k\geq \left\lceil\dfrac{x+y}{d}\right\rceil\right)
\end{array}
\right\}.
\end{equation}
\end{proposition}

\medskip

{\bf Proof.}
For a fixed pair $(x,y)$, let $\alpha$ be the mex of the set on the right-hand side of (4.3). If $z<\alpha$, then by the definition of mex, at least one of the following holds: the total number of stones $x+y+z$ is a multiple of $d$; there exists some $y'<y$ such that
$(x,y',z)\in P_1$;
or there exists some $x'<x$ such that
$(x',y,z)\in P_1$.
Hence $(x,y,z)$ cannot belong to $P_1$.

It follows that no move from $(x, y, \alpha)$ to $(x, y, z)$ with $z<q$ leads to a position in $P_1$. Moreover, by the definition of mex, none of the positions
$(x, y', \alpha)\quad (y'<y)$ or $(x', y, \alpha)\quad (x'<x)$
belongs to $P_1$. Thus, there is no move from $(x, y, \alpha)$ to a position in $P_1$.

Furthermore, $x+y+\alpha$ is not a multiple of $d$, since otherwise $\alpha$ could be written in the form
$\alpha=kd-(x+y)$.
Therefore, $(x, y, \alpha)\in P_1$,
and hence $f(x, y)=\alpha$.
\qed

\medskip
 \begin{center}
\begin{tabular}{|c||c|c|c|c|c|c|c|c|c|c|}\hline
$x\backslash y$&\ 0 \ & \ 1 \ & \ 2 \ & \ 3 \ & \ 4 \ & \ 5 \ & \ 6 \ & \ 7 \ & \ 8 \ & \ 9 \ \\ \hline\hline
0&1&0&2&3&4&6&5&7&8&9\\ \hline
1&0&1&3&2&6&5&4&8&7&11\\ \hline
2&2&3&0&1&5&4&6&9&11&7\\ \hline
3&3&2&1&0&7&8&9&4&5&6\\ \hline
4&4&6&5&7&0&2&1&3&9&8\\ \hline
5&6&5&4&8&2&1&0&10&3&12\\ \hline
6&5&4&6&9&1&0&2&11&10&3\\ \hline
7&7&8&9&4&3&10&11&0&1&2\\ \hline
8&8&7&11&5&9&3&10&1&0&4\\ \hline
9&9&11&7&6&8&12&3&2&4&0\\ \hline
\end{tabular}
\end{center}

\medskip

For $d=5$, the values of $f(x,y)$ are given in the table above.
Furthermore, the value of $f(x,y)$ can be bounded from above and below by the following inequalities.

\begin{proposition}
For every pair $(x,y)$ of nonnegative integers, the following hold.
\begin{itemize}
\item[(i)] $f(x,y)\leq x+y+1$.
\item[(ii)] $f(x,y)\geq |x-y|-1$.
\item[(iii)] If $f(x,y)=x+y+1$, then
$2(x+y)\equiv 0\pmod d$.
\end{itemize}
\end{proposition}

\medskip

{\bf Proof.}
For $d=2,3$, the value of $f(x,y)$ is explicitly given by (4.1) and (4.2), respectively, and (i), (ii), and (iii) can be verified using, for example, $\alpha\oplus\beta\leq \alpha+\beta$.

Thus, in what follows, we may assume that $d\geq 4$.
We first prove (i) and (iii) by induction on $x+y$. Terms involving $f$ with negative arguments are omitted. If $x+y=0$, that is, if $(x,y)=(0,0)$, then $f(0,0)=1$, and both (i) and (iii) hold.
Suppose that (i) and (iii) hold whenever $x+y\leq k$. We show that they also hold when $x+y=k+1$.

First, suppose, to the contrary, that
$f(x,y)\geq x+y+2$.
Since
$f(x',y')<x+y+1$ whenever $x'+y'<x+y$, the value $x+y+1$ can be excluded from $f(x,y)$ only if
$x+y+(x+y+1)$ is a multiple of $d$. Therefore, 
$x+y+(x+y)$ is not a multiple of $d$. Hence, in order for $f(x,y)\neq x+y$, we must have
$f(x-1,y)=x+y$ or $f(x,y-1)=x+y$. By the induction hypothesis (iii), applied to a pair whose sum is smaller than $x+y$, it follows that $2(x+y-1)$ is a multiple of $d$. Thus both
$2x+2y+1$ and $2x+2y-2$ are multiples of $d$. This is possible only when $d=3$, contradicting the assumption $d\geq 4$. Therefore, $f(x,y)\leq x+y+1$ also holds when $x+y=k+1$.

Next, suppose that $x+y=k+1$ and $f(x,y)=x+y+1$. 
Assume that $2(x+y)$ is not a multiple of $d$. Then, in order for $f(x,y)\neq x+y$, we must again have
$f(x-1,y)=x+y$ or $f(x,y-1)=x+y$. By the induction hypothesis (iii), it follows that $2(x+y-1)$ is a multiple of $d$.
Under this condition, $(x-1,y,x+y-1), \ (x,y-1,x+y-1)\in\mathbb{X}_d$, since
$(x-1)+y+(x+y-1)=2(x+y-1)$. Therefore, neither $f(x-1,y)$ nor $f(x,y-1)$ can be equal to $x+y-1$. Hence, in order for
$f(x,y)\neq x+y-1$, we must have either $f(x-2,y)=x+y-1$ or $f(x,y-2)=x+y-1$.
For this to occur, $2(x+y-2)$ must also be a multiple of $d$. However, both $2(x+y-1)$ and $2(x+y-2)$ can be multiples of $d$ only when $d=2$, contradicting $d\geq 4$. Therefore, when $x+y=k+1$ and $f(x,y)=x+y+1$, we must have $2(x+y)\equiv 0\pmod d$. This proves (i) and (iii) for all pairs $x,y$.

Finally, since $(x,y,f(x,y))$ is a $\CP$-position if and only if $(x,f(x,y),y)$ and $(y,f(x,y),x)$ are $\CP$-positions, applying (i) to these permutations gives $y\leq f(x,y)+x+1$ and $x\leq f(x,y)+y+1$. These inequalities imply
$f(x,y)\geq |x-y|-1$, which proves (ii).
\qed

\medskip

Building on these results, obtaining a complete solution for the three-pile case is one of our immediate goals.

\subsection{The Case $d=2^k$}

For the special case $d=2^k$, the set of $\CP$-positions in three-pile $2^k$-FTS Nim with $k\geq 3$ can be determined as follows. Recall that $P_M$ denotes the set of $\CP$-positions in mis\`ere Nim.

\begin{theorem}
Consider three-pile $2^k$-FTS Nim with $k\geq 3$.  
For $(x_1,x_2,x_3)\in \mathbb{X}$, let $\varepsilon_i$ denote the remainder of $x_i$ upon division by $2^{k-1}$.
Define $P=P_0\cup P_1$, where
\begin{eqnarray*}
P_0
&=&
\{(x_1,x_2,x_3)\in B_{2^{k-1}}\cap\mathbb{X}_{2^k}
\mid
x_i \mbox{ is a multiple of } 2^{k-1} \mbox{ for each } i\},
\\
P_1
&=&
\{(x_1,x_2,x_3)\in B_{2^{k-1}}\setminus\mathbb{X}_{2^k}
\mid
(\varepsilon_1,\varepsilon_2,\varepsilon_3)\in P_M\}.
\end{eqnarray*}
Then $P$ is the set of $\CP$-positions of $2^k$-FTS Nim.
\end{theorem}

\begin{remark}
The theorem remains valid for $k=2$, and this case is covered by Theorem 3.3, whereas the case $k=1$ is different.
\end{remark}

\medskip

{\bf Proof.}
The proof is similar in spirit to that of Theorem 3.1.
Let $m=2^{k-1}$, and write $x_i=mq_i+\varepsilon_i$, $0\leq \varepsilon_i<m$ for each $i$.

(i) We first show that if $\bm{x}\in P_0$ and
$\bm{x}\to\bm{x'}$, then $\bm{x'}\notin P_1$.
Since each $x_i$ is a multiple of $m$, we have
$\varepsilon_i=0$ for each $i$.
Hence, whenever a move $x_i\to x'_i$ is made, we have $\left\lfloor \frac{x_i}{m}\right\rfloor
\neq\left\lfloor \frac{x'_i}{m}\right\rfloor$. Since $\bm{x}\in B_m$, property (N1) of the nim-sum implies that
$\bm{x'}\notin B_m$.
Therefore, $\bm{x'}\notin P_1$.

(ii) Next, we show that if $\bm{x}\in P_1$ and
$\bm{x}\to\bm{x'}$, then $\bm{x'}\notin P_1$.
Since $\bm{x}\in P_1$, we have $q_1\oplus q_2\oplus q_3=0$ and $(\varepsilon_1,\varepsilon_2,\varepsilon_3)\in P_M$. 
Suppose, without loss of generality, that the move changes $x_1$ to $x'_1$. If $\left\lfloor \frac{x_1}{m}\right\rfloor\neq\left\lfloor \frac{x'_1}{m}\right\rfloor$, then property (N1) implies that $\bm{x'}\notin B_m$. On the other hand, if $\left\lfloor \frac{x_1}{m}\right\rfloor=\left\lfloor \frac{x'_1}{m}\right\rfloor$, 
then the quotient remains unchanged and $\varepsilon'_1<\varepsilon_1$.
Thus $(\varepsilon'_1,\varepsilon_2,\varepsilon_3)$ is obtained from
$(\varepsilon_1,\varepsilon_2,\varepsilon_3)$
by a move in mis\`ere Nim. Since
$(\varepsilon_1,\varepsilon_2,\varepsilon_3)\in P_M$,
the resulting remainder triple does not belong to $P_M$.
Therefore, in either case, $\bm{x'}\notin P_1$.

(iii) Finally, we show that if $\bm{x}\notin P$, then there exists
a position $\bm{x'}\in P_1$ such that
$\bm{x}\to\bm{x'}$. First suppose that
$q_1\oplus q_2\oplus q_3\neq 0$.
By property (N2) of the nim-sum, we can decrease one of the quotients
so that the nim-sum becomes zero. Suppose, without loss of generality,
that this is done by decreasing $q_1$ to $q'_1<q_1$.
The remainder $\varepsilon'_1$ may be chosen arbitrarily from
$\{0,1,\ldots,m-1\}$ while still satisfying $x'_1<x_1$.
We choose $\varepsilon'_1$ so that
$(\varepsilon'_1,\varepsilon_2,\varepsilon_3)\in P_M$.
Such a choice is always possible by Theorem 2.1, since $m$ is a power of $2$.

Now suppose that $q_1\oplus q_2\oplus q_3=0$.
As observed below, a remainder triple in $P_M$ has positive total
size less than $2m$; hence, since $\bm{x}\notin P$, the remainder
triple $(\varepsilon_1,\varepsilon_2,\varepsilon_3)$ does not
belong to $P_M$.
Hence, there is a move in mis\`ere Nim from this remainder triple
to a position in $P_M$.

To complete the proof, note that every position in $P_M$ with three piles, each of size
less than $m$, has a positive total number of stones strictly less
than $2m=2^k$.
Moreover, if the nim-sum of the three quotients is zero, then their
sum is even.
It follows that the position $\bm{x'}$ constructed above does not
belong to $\mathbb{X}_{2^k}$.
Thus $\bm{x'}\in P_1$ in both cases.
Therefore, $P$ is the set of $\CP$-positions of $2^k$-FTS Nim.
\qed

\medskip
%

%


The proof above also explains why the case $d=2^k$ admits such a characterization. The key point is that the quotient and remainder conditions are compatible: the quotients are governed by the ordinary nim-sum, while the remainders form a $\CP$-position of mis\`ere Nim, whose positive total size is less than $2^k$. This prevents the resulting position from having total size divisible by $2^k$. When such a separation is not available, adjusting the remainders may also affect the nim-sum of the quotients, which appears to be one of the main obstacles to obtaining a similarly simple characterization.

\section{General Forbidden-Total-Size Nim}

The $d$-FTS Nim considered in this paper is defined by forbidding moves to positions $\bm{x}$ for which the total number of stones $|\bm{x}|$ is a multiple of $d$. We now return to the original setting. Let $S$ be a subset of the nonnegative integers, and forbid moves to positions $\bm{x}$ satisfying $|\bm{x}|\in S$.
Note also that in normal-play Nim, whenever $\bm{x}$ is a $\CP$-position, $|\bm{x}|$ is always even. Hence, if $S$ contains only odd integers, the set of $\CP$-positions coincides with that of normal-play Nim.

As one example, consider the case $S=\{0,1\}$. When there are two piles, the set of $\CP$-positions is determined as follows. The proof is straightforward and is omitted.

\medskip

\begin{theorem}
Let $S=\{0,1\}$. Consider two-pile Nim in which moves to positions $\bm{x}$ satisfying $|\bm{x}|\in S$ are forbidden. Define the following subsets of $\mathbb{X}$:
\begin{eqnarray*}
P_0&=&\{(0,0), (0,1),(1,0)\},\\
P_{1,0}&=&\{(0,2),(1,1),(2,0)\},\\
P_{1,1}&=&\{(x,x)\mid x\geq 3\}.
\end{eqnarray*}
Let $P_1=P_{1,0}\cup P_{1,1}$. Then $P=P_0\cup P_1$ is the set of $\CP$-positions.
\end{theorem}

\medskip

When there are three piles under this rule, determining the set of $\CP$-positions is not straightforward. As a step toward its analysis, we establish a result analogous to that for three-pile $5$-FTS Nim. The proof is similar to that of Theorem 4.1.

\medskip

\begin{theorem}
Let $S=\{0,1\}$. Consider three-pile Nim in which moves to positions $\bm{x}$ satisfying $|\bm{x}|\in S$ are forbidden. For every pair $(x,y)$ of nonnegative integers, there exists a unique $z$ such that $(x,y,z)$ is a $\CP$-position and $(x,y,z)\notin \mathbb{X}_S$.
\end{theorem}

For each pair $(x,y)$, let $f(x,y)$ denote the value of $z$ determined in this way. Then $f(x,y)$ can be characterized as follows.

\medskip

\begin{proposition}
Let $f(0,0)=2,f(1,0)=f(0,1)=1$.
For $x+y\geq 2$, $f(x,y)$ is recursively given by
\[
f(x,y)={\rm mex}\left\{
\begin{array}{ll}
f(x',y) & (0\leq x' < x), \\
f(x,y') & (0\leq y' < y)
\end{array}
\right\}.
\]
\end{proposition}

\medskip

The values of $f(x,y)$ are shown in the table below. They can be obtained as follows. First, write $2$ in the cell corresponding to $(0,0)$, and write $1$ in the cells corresponding to $(1,0)$ and $(0,1)$. Then, similarly to the computation of Sprague--Grundy values for two-pile Nim, proceed through the cells satisfying $x+y\geq 2$ in increasing order of $x+y$, and in each cell $(x,y)$ write the smallest nonnegative integer that does not appear above it in the same column or to its left in the same row.

\medskip

\begin{center}
\begin{tabular}{|c||c|c|c|c|c|c|c|c|c|c|}\hline
$x\backslash y$&\ 0 \ & \ 1 \ & \ 2 \ & \ 3 \ & \ 4 \ & \ 5 \ & \ 6 \ & \ 7 \ & \ 8 \ & \ 9 \ \\ \hline\hline
0&2&1&0&3&4&5&6&7&8&9\\ \hline
1&1&0&2&4&3&6&5&8&7&10\\ \hline
2&0&2&1&5&6&3&4&9&10&7\\ \hline
3&3&4&5&0&1&2&7&6&9&8\\ \hline
4&4&3&6&1&0&7&2&5&11&12\\ \hline
5&5&6&3&2&7&0&1&4&12&11\\ \hline
6&6&5&4&7&2&1&0&3&13&14\\ \hline
7&7&8&9&6&5&4&3&0&1&2\\ \hline
8&8&7&10&9&11&12&13&1&0&3\\ \hline
9&9&10&7&8&12&11&14&2&3&0\\ \hline
\end{tabular}
\end{center}

\medskip

Since the values of $f(x,y)$ are determined by a structure similar to that of the Sprague--Grundy values for two-pile Nim, one may expect that they can be expressed explicitly by a closed formula involving the nim-sum. However, no such formula has been obtained at present. Nevertheless, as in Proposition 4.3, the value of $f(x,y)$ can be bounded from above and below.

\medskip

\begin{proposition}
If $x+y\geq 3$, then $|x-y|\leq f(x,y)\leq x+y$.
\end{proposition}

\medskip

The proof of this proposition is also almost the same as that of Proposition 4.3. First, the upper bound can be proved by induction on $x+y$. The lower bound then follows from the equivalence of the statements that $(x,y,f(x,y)), (x,f(x,y),y),(y,f(x,y),x)$
are $\CP$-positions.

\section*{Acknowledgements}
This work was supported by JSPS KAKENHI (Grant Number JP25K15403) and JST SPRING (Grant Number JPMJSP2115).


\end{document}